\documentclass[11pt]{article}

\usepackage{fullpage}

\usepackage[colorlinks]{hyperref}

\usepackage{xcolor}
\newcommand{\todo}[1]{{\color{red}[[}{\color{blue}\small{#1}}{\color{red}]]}}
\renewcommand{\todo}[1]{\ignorespaces}

\usepackage{amsmath, amssymb, amsthm}

\newcommand{\Z}{\mathbb{Z}}
\newcommand{\R}{\mathbb{R}}

\newcommand{\dist}{\operatorname{dist}}
\newcommand{\linspan}{\operatorname{span}}

\newcommand{\norm}[1]{\|{#1}\|}
\newcommand{\normsq}[1]{\|{#1}\|^2}

\newcommand{\lattice}{\mathcal{L}}

\newcommand{\shvec}[1][1]{\operatorname{\lambda}_{#1}}
\newcommand{\cover}{\operatorname{\mu}}
\newcommand{\smooth}[1][\epsilon]{\operatorname{\eta}_{#1}}

\newcommand{\filtr}{\mathcal{F}}
\newcommand{\filtrdetail}[1][m]{{\{\mathbf{0}\} = \lattice_0 \subset \lattice_1 \subset \dots \subset \lattice_{#1} = \lattice}}

\newcommand{\torus}[1][\lattice]{{\R^n / #1}}

\newcommand{\gaussker}[1][s]{\rho_{#1}}
\newcommand{\gaussnormsq}[1][s]{h_{\lattice, #1}}

\newcommand{\gaussembed}[1][s]{H_{\lattice, #1}}
\newcommand{\fullgaussembed}[1][s]{\gaussembed[#1]^{(k)}}
\newcommand{\finalgaussembed}[1][\lattice]{H_{#1}^{(k)}}

\newcommand{\verbosefiltrproj}[1][j]{\pi_{\filtr, #1}}
\newcommand{\verbosefiltrsuffixeq}[1][j]{\pi^\ge_{\filtr, #1}}
\newcommand{\verbosefiltrprefixne}[1][j]{\pi^<_{\filtr, #1}}
\newcommand{\verbosefiltrsuffixne}[1][j]{\pi^>_{\filtr, #1}}
\newcommand{\verbosefiltrprefixeq}[1][j]{\pi^\le_{\filtr, #1}}
\newcommand{\simplefiltrproj}[1][j]{\pi_{#1}}
\newcommand{\simplefiltrsuffixeq}[1][j]{\pi^\ge_{#1}}
\newcommand{\simplefiltrprefixne}[1][j]{\pi^<_{#1}}
\newcommand{\simplefiltrsuffixne}[1][j]{\pi^>_{#1}}
\newcommand{\simplefiltrprefixeq}[1][j]{\pi^\le_{#1}}
\newcommand{\filtrproj}{\verbosefiltrproj}
\newcommand{\filtrsuffixeq}{\verbosefiltrsuffixeq}
\newcommand{\filtrprefixne}{\verbosefiltrprefixne}
\newcommand{\filtrsuffixne}{\verbosefiltrsuffixne}
\newcommand{\filtrprefixeq}{\verbosefiltrprefixeq}

\newcommand{\verbosefiltrembed}[1][j]{E_{\filtr, \alpha, #1}}
\newcommand{\verbosefullfiltrembed}{E_{\filtr, \alpha}}
\newcommand{\simplefiltrembed}[1][j]{E_{\alpha, #1}}
\newcommand{\simplefullfiltrembed}{E_{\alpha}}
\newcommand{\filtrembed}{\verbosefiltrembed}
\newcommand{\fullfiltrembed}{\verbosefullfiltrembed}
\newcommand{\finalfiltrembed}[1][j]{F_{\lattice, #1}}
\newcommand{\finalfullfiltrembed}{F_{\lattice}}

\renewcommand{\epsilon}{\varepsilon}

\newtheorem{theorem}{Theorem}[section]
\newtheorem{lemma}[theorem]{Lemma}
\newtheorem{corollary}[theorem]{Corollary}

\newtheorem{definition}[theorem]{Definition}

\title{Nearly Optimal Embeddings of Flat Tori}

\author{
Ishan Agarwal%
\thanks{Courant Institute of Mathematical Sciences, New York University. Email: \texttt{ia1020@nyu.edu}. Research supported by National Science Foundation (NSF) under Grant No.~CCF-1814524.}
\and
Oded Regev%
\thanks{Courant Institute of Mathematical Sciences, New York University. Research supported by the Simons Collaboration on Algorithms and Geometry, a Simons Investigator Award, and by the National Science Foundation (NSF) under Grant No.~CCF-1814524.}
\and
Yi Tang%
\thanks{Courant Institute of Mathematical Sciences, New York University. Email: \texttt{yt1433@nyu.edu}.}
}

\date{}

\begin{document}

\maketitle

\begin{abstract}
We show that for any $n$-dimensional lattice $\mathcal{L}\subseteq\mathbb{R}^n$, the torus $\mathbb{R}^n/\mathcal{L}$ can be embedded into Hilbert space with $O(\sqrt{n\log n})$ distortion. This improves the previously best known upper bound of $O(n\sqrt{\log n})$ shown by Haviv and Regev (APPROX 2010) and approaches the lower bound of $\Omega(\sqrt{n})$ due to Khot and Naor (FOCS 2005, Math.\ Annal.\ 2006).
\end{abstract}

\section{Introduction}

Low distortion embeddings play an important role in many approximation algorithms, allowing one to map points in a ``difficult'' metric space into another simpler metric space (such as Hilbert space), in a way that approximately preserves distances. See the survey by Indyk~\cite{Indyk} for many examples of algorithmic applications. One interesting family of difficult metric spaces is given by \emph{flat tori}. These are defined as quotients of Euclidean space by a lattice, and play an important role in lattice problems and algorithms.

In more detail, an $n$-dimensional lattice $\lattice\subseteq\R^n$ is defined as the set of all integer linear combinations of some $n$ linearly independent vectors in $\R^n$.
The torus $\torus$ is the quotient space obtained by identifying points in $\R^n$ with each other if their difference is a lattice vector. The torus has a natural metric associated to it; namely, the distance between any two elements of the torus is defined as the minimum distance between any representative of these elements. So for instance, in the one-dimensional case $\R/\Z$, the distance between $0.1$ and $0.9$ is $0.2$.

Khot and Naor~\cite{KhotN06} considered the question of how well one can embed flat tori $\torus$ into Hilbert space. They proved that for any $\lattice$ and any embedding of $\torus$ into Hilbert space, the distortion must be at least $\Omega(\frac{\shvec(\lattice^*)}{\cover(\lattice^*)}\sqrt{n})$. Here, $\lattice^*$ is the dual lattice of $\lattice$ and $\shvec(\lattice)$ and $\cover(\lattice)$ represent the length of the shortest nonzero vector and the covering radius of $\lattice$ respectively. It is known by a result of Conway and Thompson (see \cite[Page 46]{MilnorWillard}) that, for large enough $n$, there exist lattices $\lattice$ where $\shvec(\lattice)=\cover(\lattice)$. Thus the lower bound of Khot and Naor shows that there are $n$-dimensional lattices whose torus requires distortion $\Omega(\sqrt{n})$ in any embedding into Hilbert space. In the same paper, they also present an embedding that achieves a distortion of $O(n^{3n/2})$ for any lattice $\lattice$. While the distortion of their embedding might be better than this upper bound, it is known that for some lattices it is super-polynomial~\cite[Section 7]{HavivR10}.

In~\cite{HavivR10} an $O(n\sqrt{\log n})$ distortion metric embedding is constructed, significantly reducing the gap between the upper and lower bounds. They also provide an alternative upper bound of $O(\sqrt{n\log (\cover(\lattice)/\shvec(\lattice))}$. For lattices with good geometric structure (specifically, where the ratio $\cover(\lattice)/\shvec(\lattice)$ is only polynomial) this gives an $O(\sqrt{n\log n})$ upper bound. However, in general, the ratio $\cover(\lattice)/\shvec(\lattice)$ can be arbitrarily big, in which case this alternative bound is not useful.

Our result is a nearly tight embedding of flat tori, essentially resolving the question of Khot and Naor up to a $\sqrt{\log n}$ factor.

\begin{theorem}\label{thm:main}
For any lattice $\lattice \subseteq \R^n$ there exists a metric embedding of $\torus$ into Hilbert space with distortion $O(\sqrt{n\log n})$.
\end{theorem}

\subsection{Proof overview}

Our starting point is the embedding by Haviv and Regev~\cite{HavivR10}, which is based on Gaussian measures. Their embedding achieves a distortion of $O(\sqrt{ n \log n})$ assuming that the lattice $\lattice$ has $\textrm{poly}(n)$ ``aspect ratio", i.e., the ratio between $\mu(\lattice)$
(the diameter of the torus, or equivalently, the covering
radius of the lattice) and $\lambda_1(\lattice)$ (the length of the shortest nonzero vector in the lattice) is polynomial in the dimension $n$.
Their embedding can also be applied to arbitrary lattices; the only issue is that it ``saturates" at distance $\textrm{poly}(n) \lambda_1(\lattice)$ --- points at greater distance will be contracted by the embedding.
See Section~\ref{sec:gaussian} for the details.

A natural way to address this issue is to partition the lattice into scales, and embed each scale separately. Specifically, one can define a filtration of sublattices
$\filtrdetail$ with each $\lattice_j$ capturing a different scale of the lattice. Then, for each $j=1,\ldots,m$, we project the torus on the space orthogonal to $\lattice_{j-1}$,
and embed each projection in Hilbert space separately. Our embedding is then the direct sum of the $m$ individual embeddings.

This approach does work, and is used as part of the construction in~\cite{HavivR10}. The difficulty is that it introduces an additional $\sqrt{m}$ loss in the distortion, which at worst can be $O(\sqrt{n})$ and is the reason they only achieved an overall distortion of $O(n \sqrt{\log n})$. To see where this loss comes from, consider a short vector inside the span of $\lattice_1$; this vector only contributes to the first embedding (because it becomes zero in the other $m-1$ projections). On the other hand, a short vector orthogonal to $\lattice_{m-1}$ gets accounted for in all $m$ projections, leading to an expansion of $\sqrt{m}$ (the square root due to the $L_2$ norm in the target Hilbert space).

In order to avoid this loss and achieve a $O(\sqrt{n \log n})$ distortion, it is tempting to decompose space into \emph{orthogonal} subspaces (and not nested subspaces as in the above). So instead of projecting on the subspace orthogonal to $\lattice_{j-1}$, we would like to only project on the subspace of $\lattice_{j}$ that is orthogonal to $\lattice_{j-1}$ (i.e., on the span of $\lattice_j / \lattice_{j-1}$). This, however, is impossible; projecting a lattice in such a way in general gives a dense set, and  not a lattice.\footnote{To see why, consider the two-dimensional lattice generated by $(1,0)$ and $(\pi,1)$; its projection on the first coordinate is a dense set.}

Our novel contribution is to replace this ``harsh" two-sided projection (which is in general impossible) by a more gentle ``compressed projection.'' Namely, we first project orthogonally to $\lattice_{j-1}$, and then \emph{scale down} the subspace orthogonal to $\lattice_j$. Returning to the example above, a short vector orthogonal to $\lattice_{m-1}$ is still accounted for in all $m$ ``compressed projections", but the scaling factors are such that its contributions form a geometric series, so the overall expansion is only a constant instead of $\sqrt{m}$. The technical effort is in showing that these compressions do not distort the geometry by too much; see Section~\ref{sec:embeddingtori} for details. We remark that this ``compressed projection" idea might find applications in other cases where decomposing a lattice into scales is desirable.

\section{Preliminaries}\label{prelim}

\subsection{Embeddings and Distortion}

A metric space is a tuple $(\mathcal{M},\dist_{\mathcal{M}})$ where $\mathcal{M}$ is a set and $\dist_{\mathcal{M}}:\mathcal{M}\times\mathcal{M}\to\R$ is a function such that the following hold for all $x,y,z \in \mathcal{M}$:
\begin{itemize}
    \item $\dist_{\mathcal{M}}(x,y)\ge0$, and the equality holds if and only if $x=y$,
    \item $\dist_{\mathcal{M}}(x,y)=\dist_{\mathcal{M}}(y,x)$,
    \item $\dist_{\mathcal{M}}(x,y)+\dist_{\mathcal{M}}(y,z) \ge \dist_{\mathcal{M}}(x,z)$.
\end{itemize}
For simplicity, we often write metric space $\mathcal{M}$ for $(\mathcal{M},\dist_{\mathcal{M}})$.
We also use $\dist$ without the subscript to represent the standard Euclidean metric over $\R^n$ (for some $n$ that is clear from the context).
A (metric) embedding is a mapping from one metric space to another.

\begin{definition}\label{metricdistortion}
Suppose $F: \mathcal{M}_1 \to \mathcal{M}_2$ is an embedding of metric space $\mathcal{M}_1$ into $\mathcal{M}_2$.
The distortion of $F$ is defined by
\begin{equation*}
    \inf\Bigl\{\frac{c_u}{c_l} : \forall x, y \in \mathcal{M}_1, \ c_l \cdot \dist_{\mathcal{M}_1}(x, y) \le \dist_{\mathcal{M}_2}(F(x), F(y)) \le c_u \cdot \dist_{\mathcal{M}_1}(x, y)\Bigr\}
    \;\text.
\end{equation*}
\end{definition}

\subsection{Lattices}

We now recall some standard definitions and notations regarding lattices.
A (full-rank) \emph{lattice} $\lattice \subseteq \R^{n}$ is the set of all integer linear combinations of $n$ linearly independent vectors.
This set of vectors is called a \emph{basis} of the lattice.
Equivalently, a lattice is a discrete subgroup of the additive group $\R^{n}$.
The \emph{dual lattice} $\lattice^*$ of $\lattice$ is defined as the set of all vectors $y \in \linspan(\lattice)$ such that $\langle x,y\rangle$ is an integer for all vectors $x \in \lattice$.
A \emph{sublattice} $\lattice'\subseteq \lattice$ is an additive subgroup of $\lattice$. We say that a sublattice $\lattice'\subseteq \lattice$ is \emph{primitive} if $\lattice'=\lattice\cap\linspan(\lattice')$.
All sublattices in this paper will be primitive.
For a lattice $\lattice$ and a primitive sublattice $\lattice'\subseteq\lattice$, the \emph{quotient lattice} $\lattice/\lattice'$ is defined as the projection of $\lattice$ onto the subspace orthogonal to $\linspan(\lattice')$. Sublattices and quotient lattices can be thought of as full rank while sitting inside some lower-dimensional space.
For lattice $\lattice \subseteq \R^n$, the \emph{torus} $\torus$ is naturally associated with the quotient metric, defined as
\begin{equation*}
    \dist_{\torus}(\mathbf{x}, \mathbf{y}) = \dist(\mathbf{x} - \mathbf{y}, \lattice)
    =\min_{\mathbf{v}\in \lattice} \dist(\mathbf{x} - \mathbf{y}, \mathbf{v})
    \;\text.
\end{equation*}

The \emph{length of the shortest vector} of a lattice $\lattice$, denoted by $\shvec(\lattice)$, is defined as the minimum length of a non-zero vector in $\lattice$. Note that here and elsewhere, length refers to the Euclidean norm.
The \emph{covering radius} of a lattice $\lattice$, denoted by $\cover(\lattice)$, is defined as the maximum (Euclidean) distance from any vector in $\linspan(\lattice)$ to $\lattice$.
Equivalently, as its name suggests, it is the minimum radius such that balls of that radius centered at all lattice points cover the entire $\linspan(\lattice)$.

We end with two simple technical lemmas.

\begin{lemma} \label{lem:Voronoi}
For any $n \ge 1$, lattice $\lattice \subseteq \R^n$, vectors $\mathbf{x} \in \R^n$, $\mathbf{v} \in \lattice$ such that $\norm{\mathbf{x} - \mathbf{v}} = \dist(\mathbf{x}, \lattice)$, and sublattice $\lattice' \subseteq \lattice$,
\begin{equation*}
    \norm{\pi_{\linspan(\lattice')}(\mathbf{x} - \mathbf{v})} \le \cover(\lattice')
    \;\text.
\end{equation*}
\end{lemma}

\begin{proof}
Suppose towards contradiction that $\norm{\pi_{\linspan(\lattice')}(\mathbf{x} - \mathbf{v})} > \cover(\lattice')$.
Then consider the lattice point $\mathbf{u} \in \lattice'$ that is a closest lattice point to $\pi_{\linspan(\lattice')}(\mathbf{x} - \mathbf{v})$ in $\lattice'$. By definition $\norm{\pi_{\linspan(\lattice')}(\mathbf{x} - \mathbf{v}) - \mathbf{u}} \le \cover(\lattice')$.
Observe that
\begin{equation*}
    \begin{aligned}
        \normsq{\mathbf{x} - (\mathbf{v} + \mathbf{u})}
        &= \normsq{\pi_{\linspan(\lattice')}(\mathbf{x} - \mathbf{v} - \mathbf{u})} + \normsq{\pi_{\linspan(\lattice / \lattice')}(\mathbf{x} - \mathbf{v} - \mathbf{u})} \\
        &= \normsq{\pi_{\linspan(\lattice')}(\mathbf{x} - \mathbf{v}) - \mathbf{u}} + \normsq{\pi_{\linspan(\lattice / \lattice')}(\mathbf{x} - \mathbf{v})} \\
        &< \normsq{\pi_{\linspan(\lattice')}(\mathbf{x} - \mathbf{v})} + \normsq{\pi_{\linspan(\lattice / \lattice')}(\mathbf{x} - \mathbf{v})} \\
        &= \normsq{\mathbf{x} - \mathbf{v}}
        \;\text,
    \end{aligned}
\end{equation*}
which contradicts with the fact that $\norm{\mathbf{x} - \mathbf{v}} = \dist(\mathbf{x}, \lattice) = \min_{\mathbf{v}' \in \lattice} \norm{\mathbf{x} - \mathbf{v}'}$.
\end{proof}

\begin{lemma} \label{lem:covering-additive}
For any lattice $\lattice$ and sublattice $\lattice' \subseteq \lattice$,
\begin{equation*}
    \cover(\lattice)^2 \le \cover(\lattice')^2 + \cover(\lattice / \lattice')^2
    \;\text.
\end{equation*}
\end{lemma}

\begin{proof}
For any $\mathbf{x} \in \linspan(\lattice)$, let $\mathbf{v} \in \lattice$ be a lattice point such that
\begin{equation*}
    \norm{\pi_{\linspan(\lattice / \lattice')}(\mathbf{x} - \mathbf{v})} = \dist(\pi_{\linspan(\lattice / \lattice')}(\mathbf{x}), \lattice / \lattice')
    \;\text.
\end{equation*}
Without loss of generality it can be assumed that $\norm{\pi_{\linspan(\lattice')}(\mathbf{x} - \mathbf{v})} \le \cover(\lattice')$
(since otherwise, we can use $\mathbf{v} + \mathbf{u}$ instead of $\mathbf{v}$, where $\mathbf{u}$ is a closest lattice point to $\pi_{\linspan(\lattice')}(\mathbf{x} - \mathbf{v})$ in $\lattice'$).
Then
\begin{equation*}
    \begin{aligned}
        \dist(\mathbf{x}, \lattice)^2
        &\le \normsq{\mathbf{x} - \mathbf{v}} \\
        &= \normsq{\pi_{\linspan(\lattice')}(\mathbf{x} - \mathbf{v})} + \normsq{\pi_{\linspan(\lattice / \lattice')}(\mathbf{x} - \mathbf{v})} \\
        &\le \cover(\lattice')^2 + \cover(\lattice / \lattice')^2
        \;\text.
    \end{aligned}
\end{equation*}
The bound holds for any vector $\mathbf{x}$.
Hence $\cover(\lattice)^2 \le \cover(\lattice')^2 + \cover(\lattice / \lattice')^2$, as desired.
\end{proof}

\section{Embedding Tori into Hilbert Space}\label{sec:gaussian}

The goal of this section is to prove Lemma~\ref{lem:gaussian-distortion}, which summarizes the properties of the Gaussian embedding
from~\cite{HavivR10}, including a modified contraction property which we make explicit (see left-hand side of~\eqref{eq:distortionbasicembedding}). The proof closely follows that  of~\cite[Theorem 1.4]{HavivR10}.
We start with some preliminary definitions and results from~\cite{HavivR10}.

For $s>0$ and $\mathbf{x} \in \R^n$ we define  $\gaussker(\mathbf{x}) = \exp(-\pi \normsq{\mathbf{x} / s})$. For any discrete set $A$, its Gaussian mass $\rho_s(A)$ is defined as $\sum_{\mathbf{x}\in A}\gaussker(\mathbf{x})$.
The \emph{smoothing parameter} of a lattice $\lattice$ is defined with respect to an $\epsilon>0$ and is given by
\begin{equation*}
    \smooth(\lattice)=\min\{s:\rho_{1/s}(\lattice^*)\leq1+\epsilon\}
    \;\text.
\end{equation*}

\begin{lemma}[{\cite[Lemma 2.5]{HavivR10}}] \label{lem:smoothing}
For any $n \ge 1$ and lattice $\lattice \subseteq \R^n$,
$\smooth(\lattice^*) \le \frac{2 \sqrt{n}}{\shvec(\lattice)}$ where $\epsilon = 2^{-10 n}$.
\end{lemma}

Consider the function
\begin{equation*}
    \gaussnormsq(\mathbf{x}) = 1 - \frac{\gaussker(\lattice - \mathbf{x})}{\gaussker(\lattice)}
    \;\text.
\end{equation*}
Below we list some basic properties of this function,
which ideally we would like to be proportional to the squared distance from the lattice.
This is indeed the case, assuming the distance is not too large compared to $s$, and that $s$ itself is small compared to the geometry of the lattice. The upper bound is shown in Item 1, and the lower bound is established
in Items 2 and 3 (which give very similar bounds).
When the distance is sufficiently larger than $s$, the function reaches saturation, as shown in Item 4.

\todo{OR: we should think if we can merge Items 2 and 3 into one item. Maybe go back to HR and try to see if we prove one statement? It's ugly to have two, and they are extremely similar}

\begin{lemma}[{\cite[Lemmas 3.1 and 3.2]{HavivR10}}] \label{lem:gaussian-bound}
For any $n \ge 1$, lattice $\lattice \subseteq \R^n$, $s > 0$, and vector $\mathbf{x} \in \R^n$,
\begin{enumerate}
    \item $s^2 \cdot \gaussnormsq(\mathbf{x}) \le \pi \cdot \dist(\mathbf{x}, \lattice)^2$,
    \item $s^2 \cdot \gaussnormsq(\mathbf{x}) \ge c \cdot \dist(\mathbf{x}, \lattice)^2$ if $s \le \frac{1}{2 \smooth(\lattice^*)}$ for some $0 < \epsilon \le \frac{1}{1000}$ and $\dist(\mathbf{x}, \lattice) \le \frac{s}{\sqrt{2}}$, where $c$ is an absolute constant,
    \item $\gaussnormsq(\mathbf{x}) \ge 1 - e^{-\pi \dist(\mathbf{x}, \lattice)^2 / s^2} - 2^{-11 n}$ if $\shvec(\lattice) \ge 4 \sqrt{n} \cdot s$,
    \item $\gaussnormsq(\mathbf{x}) \ge 1 - 2^{-11 n}$ if $\dist(\mathbf{x}, \lattice) > 2 \sqrt{n} \cdot s$.
\end{enumerate}
\end{lemma}

\begin{definition}[{\cite[Section 5]{HavivR10}}] \label{def-embedding-gaussian-component}
For lattice $\lattice \subseteq \R^n$ and $s > 0$, the embedding $\gaussembed: \torus \to L_2(\torus)$ maps vector $\mathbf{x} \in \R^n$ to the function $\gaussembed(\mathbf{x}) \in L_2(\torus)$ given by
\begin{equation*}
    \gaussembed(\mathbf{x})(\mathbf{y}) = \frac{s}{\sqrt{2 \rho_s(\lattice)}} \left( \frac{2}{s} \right)^{n / 2} \rho_{\frac{s}{\sqrt{2}}}(\lattice + \mathbf{y} - \mathbf{x})
    \;\text.
\end{equation*}
\end{definition}

\begin{lemma}[{\cite[Proposition 5.1]{HavivR10}}] \label{lem:gaussian-distance}
For any $n \ge 1$, lattice $\lattice \subseteq \R^n$, $s > 0$, and vectors $\mathbf{x}, \mathbf{y} \in \R^n$,
$\dist_{L_2(\torus)}(\gaussembed(\mathbf{x}), \gaussembed(\mathbf{y}))^2 = s^2 \cdot \gaussnormsq(\mathbf{x} - \mathbf{y})$.
\end{lemma}

\begin{definition}[{\cite[Section 5.1]{HavivR10}}] \label{def-embedding-gaussian}
For lattice $\lattice \subseteq \R^n$, $s > 0$, and $k \ge 1$, the embedding $\fullgaussembed$ is defined by $\fullgaussembed = (\gaussembed[s_1], \dots, \gaussembed[s_k])$ where $s_i = 2^{i-1} s$.
We often take $s = \shvec(\lattice)/(4 \sqrt{n})$, in which case we omit the subscript $s$ and simply write
$\finalgaussembed$.
\end{definition}

\begin{lemma} \label{lem:gaussian-distortion}
For any $n \ge 1$, lattice $\lattice \subseteq \R^n$, $k \ge 1$, and vectors $\mathbf{x}, \mathbf{y} \in \R^n$,
\begin{align}
    \frac{c_H}{n} \cdot \min(\dist_{\torus}(\mathbf{x}, \mathbf{y}), 2^{k-1} \shvec(\lattice))^2 \le \dist_{L_2(\torus)^k}(\finalgaussembed(\mathbf{x}), \finalgaussembed(\mathbf{y}))^2 \le \pi k \cdot \dist_{\torus}(\mathbf{x}, \mathbf{y})^2
    \;\text,
    \label{eq:distortionbasicembedding}
\end{align}
where $c_H>0$ is an absolute constant.
\todo{YT: note for posterity: if not using Lemma~\ref{lem:gaussian-bound} Item~3, there would be a lower bound on $k$ (roughly $k \ge \log \sqrt{n}$).}
\end{lemma}

\begin{proof}
By Lemma~\ref{lem:gaussian-distance}, and recalling the notation $s_i = 2^{i-1} s$ where $s = \frac{\shvec(\lattice)}{4 \sqrt{n}}$,
\begin{equation*}
    \begin{aligned}
        \dist_{L_2(\torus)^k}(\finalgaussembed(\mathbf{x}), \finalgaussembed(\mathbf{y}))^2
        &= \sum_{i = 1}^k \dist_{L_2(\torus)}(\gaussembed[s_i](\mathbf{x}), \gaussembed[s_i](\mathbf{y}))^2 \\
        &= \sum_{i = 1}^k s_i^2 \cdot \gaussnormsq[s_i](\mathbf{x} - \mathbf{y})
        \;\text.
    \end{aligned}
\end{equation*}
Noting that $\dist_{\torus}(\mathbf{x}, \mathbf{y}) = \dist(\mathbf{x} - \mathbf{y}, \lattice)$,
the upper bound in~\eqref{eq:distortionbasicembedding} follows from Item~1 in Lemma~\ref{lem:gaussian-bound}:
\begin{align*}
    \sum_{i = 1}^k s_i^2 \cdot \gaussnormsq[s_i](\mathbf{x} - \mathbf{y})
    \le \sum_{i = 1}^k \pi \cdot \dist(\mathbf{x} - \mathbf{y}, \lattice)^2
    = \pi k \cdot \dist(\mathbf{x} - \mathbf{y}, \lattice)^2
    \;\text.
\end{align*}

For the lower bound in~\eqref{eq:distortionbasicembedding}, we will show that for any $\mathbf{x}, \mathbf{y} \in \R^n$, there exists $i \in \{1, \dots, k\}$ such that
    \begin{align}
        s_i^2 \cdot \gaussnormsq[s_i](\mathbf{x} - \mathbf{y}) \ge \frac{c_H}{n} \cdot \min(\dist(\mathbf{x} - \mathbf{y}, \lattice), 2^{k-1} \shvec(\lattice))^2
        \;\text.
        \label{eq:goalforlowerboundondistortion}
    \end{align}
We consider three cases.
\begin{enumerate}
    \item $\dist(\mathbf{x} - \mathbf{y}, \lattice) \le \frac{\shvec(\lattice)}{4 \sqrt{2 n}} = \frac{s}{\sqrt{2}}$.
    Note that according to Lemma~\ref{lem:smoothing},
    $s \le \frac{1}{2 \smooth(\lattice^*)}$ for some $0 < \epsilon \le \frac{1}{1000}$.
    Then by Item~2 of Lemma~\ref{lem:gaussian-bound},
    \begin{equation*}
        s^2 \cdot \gaussnormsq(\mathbf{x} - \mathbf{y})
        \ge c \cdot \dist(\mathbf{x} - \mathbf{y}, \lattice)^2
        \ge \frac{c}{n} \cdot \dist(\mathbf{x} - \mathbf{y}, \lattice)^2
        \;\text,
    \end{equation*}
    which proves~\eqref{eq:goalforlowerboundondistortion} with $i = 1$.

    \item $\frac{s}{\sqrt{2}} < \dist(\mathbf{x} - \mathbf{y}, \lattice) \le \shvec(\lattice) = 4 \sqrt{n} \cdot s$.
    By Item~3 of Lemma~\ref{lem:gaussian-bound},
    \begin{equation*}
        \begin{aligned}
            s^2 \cdot \gaussnormsq(\mathbf{x} - \mathbf{y})
            &\ge s^2 \cdot (1 - e^{-\pi / 2} - 2^{-11 n}) \\
            &= \frac{1 - e^{-\pi / 2} - 2^{-11 n}}{16 n} \cdot \shvec(\lattice)^2 \\
            &\ge \frac{1 - e^{-\pi / 2} - 2^{-11 n}}{16 n} \cdot \dist(\mathbf{x} - \mathbf{y}, \lattice)^2
            \;\text,
        \end{aligned}
    \end{equation*}
    which again proves~\eqref{eq:goalforlowerboundondistortion} with $i = 1$.

    \item $\dist(\mathbf{x} - \mathbf{y}, \lattice) > 4 \sqrt{n} \cdot s$.
    Let $j \in \{2, \dots, k\}$ be the largest index such that $2 \sqrt{n} \cdot s_j < \dist(\mathbf{x} - \mathbf{y}, \lattice)$.
    Notice that if $j<k$ then
    $4 \sqrt{n} \cdot s_j = 2 \sqrt{n} \cdot s_{j+1} \ge \dist(\mathbf{x} - \mathbf{y}, \lattice)$, and that if $j=k$,
    then $4 \sqrt{n} \cdot s_j = 2^{k-1} \shvec(\lattice)$.
    Then by Item~4 of Lemma~\ref{lem:gaussian-bound},
    \begin{equation*}
        \begin{aligned}
            s_j^2 \cdot \gaussnormsq[s_j](\mathbf{x} - \mathbf{y})
            &\ge s_j^2 \cdot (1 - 2^{-11 n}) \\
            &= \frac{1 - 2^{-11 n}}{16 n} \cdot (4 \sqrt{n} \cdot s_j)^2 \\
            &\ge \frac{1 - 2^{-11 n}}{16 n} \cdot \min(\dist(\mathbf{x} - \mathbf{y}, \lattice),
            2^{k-1} \shvec(\lattice))^2
            \;\text,
        \end{aligned}
    \end{equation*}
    which proves~\eqref{eq:goalforlowerboundondistortion} with $i = j$.
    \qedhere
\end{enumerate}
\end{proof}

\section{Embedding into Tori}\label{sec:embeddingtori}

The goal of this section is to prove Lemma~\ref{lem:filtration-distortion}, which shows that there exists an embedding from an arbitrary torus into a tuple of tori with good geometry.
The embedding is constructed based on ``good filtrations,'' which we define and instantiate in Section~\ref{sec:goodfilter}.
The definition of the embedding is given in Section~\ref{main_embedding}, and its expansion and contraction properties are shown in Section~\ref{sec:expansion} and Section~\ref{sec:contraction} respectively.
The contraction property matches the modified notion of contraction used in Lemma~\ref{lem:gaussian-distortion}.

\subsection{Good Filtrations}\label{sec:goodfilter}

In this section we define the notion of a \emph{$(q,\gamma)$-filtration} (Definition~\ref{goodfiltr}) and show how to construct a good one for every lattice (Lemma~\ref{lem:existsgoodfilter}).
We also include a small technical lemma that will be useful later (Lemma~\ref{lem:filtration-prefix}).

A \emph{filtration} of a lattice $\lattice$ is a chain of sublattices $\filtrdetail$.
We call $m$ the \emph{size} of the filtration.

\begin{definition} \label{goodfiltr}
For $q \ge 1$, $\gamma > 1$,
we say that a filtration $\filtrdetail$ is a \emph{$(q, \gamma)$-filtration} if it satisfies both
\begin{enumerate}
    \item $\cover(\lattice_j / \lattice_{j-1}) \le q \shvec(\lattice_j / \lattice_{j-1}) / 2$ for all $1 \le j \le m$, and
    \item $\shvec(\lattice_{j+1} / \lattice_j) \ge \gamma \shvec(\lattice_j / \lattice_{j-1})$ for all $1 \le j < m$.
\end{enumerate}
\end{definition}

Our construction of good filtrations is based on Korkine-Zolotarev bases, defined next. Recall first that for a sequence of vectors $(\mathbf{b}_1, \dots, \mathbf{b}_n)$, its \emph{Gram-Schmidt orthogonalization} $(\mathbf{b}'_1, \dots, \mathbf{b}'_n)$ is defined by
\begin{equation*}
    \mathbf{b}'_i = \mathbf{b}_i - \sum_{j=1}^{i-1} \mu_{i,j} \mathbf{b}'_j
    \;\text, \quad \text{where} \
    \mu_{i,j} = \frac{\langle \mathbf{b}_i, \mathbf{b}'_j \rangle}{\langle \mathbf{b}'_j, \mathbf{b}'_j\rangle}
    \;\text,
\end{equation*}
i.e., $\mathbf{b}'_i$ is the projection of $\mathbf{b}_i$ on the space orthogonal to $\linspan(\mathbf{b}_1,\ldots,\mathbf{b}_{i-1})$.

\begin{definition} \label{def:kzbasis}
A basis $(\mathbf{b}_1,\dots,\mathbf{b}_n)$
for a lattice $\lattice \subseteq \R^n$
is called a \emph{Korkine-Zolotarev basis} if
\begin{itemize}
    \item $\mathbf{b}'_i$ is a shortest vector of $\lattice/\lattice_{i-1}$ for all $1 \le i \le n$, and
    \item $|\mu_{i,j}| \le 1/2$ for all $1 \le j < i \le n$,
\end{itemize}
where $(\mathbf{b}'_1, \dots, \mathbf{b}'_n)$ is the Gram-Schmidt orthogonalization of $(\mathbf{b}_1, \dots, \mathbf{b}_n)$, $\mu_{i,j}$ are the corresponding coefficients, and $\lattice_i$ is the lattice generated by $(\mathbf{b}_1,\dots,\mathbf{b}_i)$ (with $\lattice_0 = \{\mathbf{0}\}$).
\end{definition}

It is easy to prove that a Korkine-Zolotarev basis exists for any lattice.
We remark that the second property above will not be used in this paper.

\begin{lemma} \label{lem:existsgoodfilter}
For any $n \ge 1$, lattice $\lattice \subseteq \R^n$, and $\gamma > 1$,
there exists a $(\gamma \sqrt{n}, \gamma)$-filtration of $\lattice$.
\end{lemma}

\begin{proof}
Let $(\mathbf{b}_1,\ldots,\mathbf{b}_n)$
be a Korkine-Zolotarev basis of $\lattice$.
Let $(\mathbf{b}'_1, \dots, \mathbf{b}'_n)$ be
its Gram-Schmidt orthogonalization, and consider
the filtration $\filtrdetail[n]$ where $\lattice_i$ is the lattice generated by $(\mathbf{b}_1,\ldots,\mathbf{b}_i)$.
From Definition~\ref{def:kzbasis} we know that $\shvec(\lattice / \lattice_{i-1}) = \norm{\mathbf{b}'_i} = \shvec(\lattice_k / \lattice_{i-1})$, for all $1 \le i \le k \le n$.
Construct a coarsening of this filtration, $\{\mathbf{0}\} = \lattice_{i_0} \subset \lattice_{i_1} \subset \dots \subset \lattice_{i_m} = \lattice$, as follows. Let $i_0 = 0$ and for $j \ge 1 $, $i_j \in \{i_{j-1}+1, \dots, n\}$ be the largest index such that $\norm{\mathbf{b}'_k} \le \gamma \norm{\mathbf{b}'_{i_{j-1}+1}}$ for all $k \in \{i_{j-1}+1, \dots, i_j\}$. Finally, stop when $i_m = n$.
We are going to show that this coarser filtration is a $(\gamma \sqrt{n}, \gamma)$-filtration.

We observe that for all $1 \le j \le m$,
\begin{equation*}
    \shvec(\lattice_{i_j} / \lattice_{i_{j-1}})
    = \norm{\mathbf{b}'_{i_{j-1}+1}}
    \;\text.
\end{equation*}
Then, by construction of the coarsening, for all $1 \le j < m$,
\begin{equation*}
    \shvec(\lattice_{i_{j+1}} / \lattice_{i_j})
    = \norm{\mathbf{b}'_{i_j+1}}
    > \gamma \norm{\mathbf{b}'_{i_{j-1}+1}}
    = \gamma \shvec(\lattice_{i_j} / \lattice_{i_{j-1}})
    \;\text.
\end{equation*}
This proves the second property of a $(\gamma \sqrt{n}, \gamma)$-filtration.
Moreover,
\begin{equation*}
    \begin{aligned}
        \cover(\lattice_{i_j} / \lattice_{i_{j-1}})^2
        &\le \sum_{k = i_{j-1}+1}^{i_j} \cover(\lattice_k / \lattice_{k-1})^2 \\
        &= \sum_{k = i_{j-1}+1}^{i_j} \normsq{\mathbf{b}'_k} / 4 \\
        &\le \sum_{k = i_{j-1}+1}^{i_j} \gamma^2 \normsq{\mathbf{b}'_{i_{j-1}+1}} / 4 \\
        &\le \gamma^2 n \cdot \shvec(\lattice_{i_j} / \lattice_{i_{j-1}})^2 / 4
        \;\text,
    \end{aligned}
\end{equation*}
where the first inequality is by Lemma~\ref{lem:covering-additive} and the second inequality is by construction of the coarsening. This proves the first property of a $(\gamma \sqrt{n}, \gamma)$-filtration.
\end{proof}

We end by proving a small property of $(q, \gamma)$-filtrations.

\begin{lemma} \label{lem:filtration-prefix}
For any $(q, \gamma)$-filtration $\filtrdetail$ and $1 \le j \le m$,
\begin{equation*}
    \cover(\lattice_j) \le \frac{q}{\sqrt{1 - 1 / \gamma^2}} \cdot \shvec(\lattice_j / \lattice_{j-1}) / 2
    \;\text.
\end{equation*}
Consequently, if $\gamma \ge 2$, then $ \cover(\lattice_j) \le q \shvec(\lattice_j / \lattice_{j-1})$.
\end{lemma}

\begin{proof}
The inequality can be proved as follows:
\begin{equation*}
    \begin{aligned}
        \cover^2(\lattice_j)
        &\le \sum_{i = 1}^j \cover^2(\lattice_i / \lattice_{i-1}) \\
        &\le \sum_{i = 1}^j q^2 \shvec^2(\lattice_i / \lattice_{i-1}) / 4 \\
        &\le \sum_{i = 1}^j \frac{q^2}{\gamma^{2 (j - i)}} \cdot \shvec^2(\lattice_j / \lattice_{j-1}) / 4 \\
        &\le \frac{q^2}{1 - 1 / \gamma^2} \cdot \shvec^2(\lattice_j / \lattice_{j-1}) / 4
        \;\text,
    \end{aligned}
\end{equation*}
where the first inequality uses Lemma~\ref{lem:covering-additive}, the second inequality uses the first property in Definition~\ref{goodfiltr}, and the third inequality uses the second property in Definition~\ref{goodfiltr}.
\end{proof}

\subsection{The Embedding}\label{main_embedding}

Let $\filtr$ be a filtration $\filtrdetail$ of a lattice $\lattice \subseteq \R^n$.
The filtration naturally induces an orthogonal decomposition of $\R^n$ into $m$ subspaces, namely, $\linspan(\lattice_j / \lattice_{j-1})$ for $j=1,\ldots,m$.
We use $\filtrproj$ to denote $\pi_{\linspan(\lattice_j / \lattice_{j-1})}$, the projection on the $j$-th subspace. We will similarly use $\filtrsuffixeq$, $\filtrprefixne$, $\filtrsuffixne$, and $\filtrprefixeq$ to denote projections on the span of prefixes and suffixes of this decomposition. Specifically, for $1 \le j \le m$ we have
$\filtrsuffixeq = \pi_{\linspan(\lattice / \lattice_{j-1})}$,
$\filtrprefixne = \pi_{\linspan(\lattice_{j-1})}$,
$\filtrsuffixne = \pi_{\linspan(\lattice / \lattice_j)}$, and
$\filtrprefixeq = \pi_{\linspan(\lattice_j)}$.

\begin{definition} \label{def-embedding-lattice-component}
For filtration $\filtr$ of size $m$, $0 < \alpha < 1$, and $1 \le j \le m$, the embedding $\filtrembed$ is defined by
\begin{equation*}
    \filtrembed(\mathbf{x}) = \sum_{i = j}^m \alpha^{i - j} \filtrproj[i](\mathbf{x})
    \;\text.
\end{equation*}
\end{definition}

Note that since $\filtrembed$ is linear, for any lattice $\lattice \subseteq \R^n$ and vector $\mathbf{x} \in \R^n$, $\filtrembed(\mathbf{x} + \lattice) = \filtrembed(\mathbf{x}) + \filtrembed(\lattice)$, and thus $\filtrembed$ is a well-defined embedding from the torus $\torus$ to the torus $\filtrembed(\torus)$.

\begin{definition} \label{def-embedding-lattice}
For a filtration $\filtr$ of size $m$ and $0 < \alpha < 1$, the embedding $\fullfiltrembed$ is defined by $\fullfiltrembed = (\filtrembed[1], \dots, \filtrembed[m])$
with the metric being $\ell_2$ of the tori metrics.
\end{definition}

\subsection{Expansion of the Embedding} \label{sec:expansion}

\begin{definition}[Realization of distance in torus]
For any $n \ge 1$, lattice $\lattice \subseteq \R^n$, and vector $\mathbf{x} \in \R^n$, since $\dist(\mathbf{x}, \lattice) = \min_{\mathbf{v} \in \lattice} \norm{\mathbf{x} - \mathbf{v}}$, there always exists $\mathbf{v} \in \lattice$ such that $\dist(\mathbf{x}, \lattice) = \norm{\mathbf{x} - \mathbf{v}}$. We say such minimizer $\mathbf{v}$ \emph{realizes} the distance $\dist(\mathbf{x}, \lattice)$. Similarly, for vectors $\mathbf{x}, \mathbf{y} \in \R^n$, we say $\mathbf{v}$ \emph{realizes} the distance $\dist_{\torus}(\mathbf{x}, \mathbf{y})$ if $\dist_{\torus}(\mathbf{x}, \mathbf{y}) = \norm{\mathbf{x} - \mathbf{y} - \mathbf{v}}$.
\end{definition}

\begin{lemma}[Expansion of the embedding] \label{lem:filtration-upper}
For any $n \ge 1$, lattice $\lattice \subseteq \R^n$ with filtration $\filtr$, $0 < \alpha < 1$, and vectors $\mathbf{x}, \mathbf{y} \in \R^n$,
\begin{align*}
    \dist_{\fullfiltrembed(\torus)}(\fullfiltrembed(\mathbf{x}), \fullfiltrembed(\mathbf{y}))^2
    &:= \sum_{j = 1}^m \dist_{\filtrembed(\torus)}(\filtrembed(\mathbf{x}), \filtrembed(\mathbf{y}))^2   \\
    &\le
    \frac{1}{1 - \alpha^2} \cdot \dist_{\torus}(\mathbf{x}, \mathbf{y})^2
    \;\text.
\end{align*}
\end{lemma}

\begin{proof}
Let $m$ be the size of $\filtr$.
For all $\mathbf{v} \in \lattice$,
the embedded distance can be bounded from above by
\begin{equation*}
    \begin{aligned}
        \sum_{j = 1}^m \dist_{\filtrembed(\torus)}(\filtrembed(\mathbf{x}), \filtrembed(\mathbf{y}))^2
        &\le \sum_{j = 1}^m \normsq{\filtrembed(\mathbf{x} - \mathbf{y} - \mathbf{v})} \\
        &= \sum_{j = 1}^m \sum_{i = j}^m \alpha^{2(i - j)} \normsq{\filtrproj[i](\mathbf{x} - \mathbf{y} - \mathbf{v})} \\
        &\le \frac{1}{1 - \alpha^2} \cdot \sum_{i = 1}^m \normsq{\filtrproj[i](\mathbf{x} - \mathbf{y} - \mathbf{v})} \\
        &= \frac{1}{1 - \alpha^2} \cdot \normsq{\mathbf{x} - \mathbf{y} - \mathbf{v}}
        \;\text,
    \end{aligned}
\end{equation*}
which, for $\mathbf{v}$ realizing $\dist_{\torus}(\mathbf{x}, \mathbf{y})$, gives $\frac{1}{1 - \alpha^2} \cdot \dist_{\torus}(\mathbf{x}, \mathbf{y})^2$ as desired.
\end{proof}

\subsection{Contraction of the Embedding} \label{sec:contraction}

\begin{lemma}\label{lem:closest-change}
For any $n \ge 1$, lattice $\lattice \subseteq \R^n$, lattice point $\mathbf{v}' \in \lattice$ realizing $\dist(\mathbf{x}, \lattice)$, and
lattice point $\mathbf{v} \in \lattice$,
\begin{equation*}
    \norm{\mathbf{x} - \mathbf{v}} \ge \frac{1}{2} \norm{\mathbf{v} - \mathbf{v}'}
    \;\text.
\end{equation*}
Consequently, if $\mathbf{v}$ does not realize $\dist(\mathbf{x}, \lattice)$, then $\mathbf{v} \ne \mathbf{v}'$ and
\begin{equation*}
    \norm{\mathbf{x} - \mathbf{v}} \ge \frac{1}{2} \shvec(\lattice)
    \;\text.
\end{equation*}
\end{lemma}

\begin{proof}
By definition, $\norm{\mathbf{x} - \mathbf{v}'} \le \norm{\mathbf{x} - \mathbf{v}}$.
Then by the triangle inequality, $\norm{\mathbf{v} - \mathbf{v}'} \le \norm{\mathbf{x} - \mathbf{v}} + \norm{\mathbf{x} - \mathbf{v}'} \le 2 \norm{\mathbf{x} - \mathbf{v}}$, as desired.
\end{proof}

\begin{lemma} \label{lem:shortest-untouched}
For any $(q, \gamma)$-filtration $\filtr$ given by $\filtrdetail$, $1 / \gamma \le \alpha < 1$, and $1 \le j \le m$,
$\shvec(\filtrembed(\lattice)) = \shvec(\lattice_j / \lattice_{j-1})$.
\end{lemma}

\begin{proof}
We prove the claim by induction on $j$. When $j = m$, $\filtrembed[m](\lattice) = \lattice_m / \lattice_{m-1}$, and thus the claim holds trivially.

Suppose the claim holds for $j+1$.
Then for $j$, note that $\lattice_j / \lattice_{j-1} \subseteq \filtrembed(\lattice)$. Therefore $\shvec(\filtrembed(\lattice))$ is the minimum of $\shvec(\lattice_j / \lattice_{j-1})$ and the minimum length of vectors in the set $\filtrembed(\lattice) \setminus (\lattice_j / \lattice_{j-1})$.
Since $\filtrembed = \filtrproj + \alpha \filtrembed[j+1]$, the length of any vector in $\filtrembed(\lattice) \setminus (\lattice_j / \lattice_{j-1})$ is bounded from below by
\begin{equation*}
    \begin{aligned}
        \alpha \shvec(\filtrembed[j+1](\lattice))
        &= \alpha \shvec(\lattice_{j+1} / \lattice_j) \\
        &\ge \alpha \gamma \shvec(\lattice_j / \lattice_{j-1}) \\
        &\ge \shvec(\lattice_j / \lattice_{j-1})
        \;\text,
    \end{aligned}
\end{equation*}
where the equality is the induction assumption and the first inequality uses the second property in Definition~\ref{goodfiltr}.
Hence $\shvec(\filtrembed(\lattice)) = \shvec(\lattice_j / \lattice_{j-1})$, as desired.
\end{proof}

Combining Lemma~\ref{lem:shortest-untouched} with Lemma~\ref{lem:filtration-prefix} as well as Definition~\ref{goodfiltr}, we immediately get the following corollary.

\begin{corollary} \label{cor:filtr}
For any $(q, \gamma)$-filtration $\filtr$ given by $\filtrdetail$ with $\gamma \ge 2$ and $1 / \gamma \le \alpha < 1$,
\begin{enumerate}
    \item $\cover(\lattice_j) \le q \shvec(\filtrembed(\lattice))$ for all $1 \le j \le m$, and
    \item $\shvec(\filtrembed[j+1](\lattice)) \ge \gamma \shvec(\filtrembed(\lattice))$ for all $1 \le j < m$.
\end{enumerate}
\end{corollary}

\begin{lemma}[Contraction of the embedding] \label{lem:filtration-lower}
For any $n \ge 1$, lattice $\lattice \subseteq \R^n$ with $(q, \gamma)$-filtration $\filtr$ of size $m$ satisfying $\gamma \ge 2$ and $q \le \gamma^2 / 32$, $\frac{1}{2} \le \alpha < 1$, and vectors $\mathbf{x}, \mathbf{y} \in \R^n$,
\begin{equation} \label{eqn:lower}
    \sum_{j = 1}^m \min \bigl( \dist_{\filtrembed(\torus)}(\filtrembed(\mathbf{x}), \filtrembed(\mathbf{y})), q^2 \shvec(\filtrembed(\lattice)) \bigr)^2 \ge c_E \cdot \dist_{\torus}(\mathbf{x}, \mathbf{y})^2
    \;\text,
\end{equation}
where $c_E>0$ is an absolute constant.
\end{lemma}

\begin{proof}
For simplicity, we omit the subscript $\filtr$ in the notations $\filtrproj$, $\filtrsuffixeq$, $\filtrprefixne$, $\filtrsuffixne$, $\filtrprefixeq$, $\filtrembed$ and $\fullfiltrembed$ in this proof.

\renewcommand{\filtrproj}{\simplefiltrproj}
\renewcommand{\filtrsuffixeq}{\simplefiltrsuffixeq}
\renewcommand{\filtrprefixne}{\simplefiltrprefixne}
\renewcommand{\filtrsuffixne}{\simplefiltrsuffixne}
\renewcommand{\filtrprefixeq}{\simplefiltrprefixeq}
\renewcommand{\filtrembed}{\simplefiltrembed}
\renewcommand{\fullfiltrembed}{\simplefullfiltrembed}

Let $\mathbf{v} \in \lattice$ be a lattice point that realizes $\dist_{\torus}(\mathbf{x}, \mathbf{y})$. Then $\dist_{\torus}(\mathbf{x}, \mathbf{y}) = \norm{\mathbf{x} - \mathbf{y} - \mathbf{v}}$.
Hence our goal is equivalently to show that the left-hand side of \eqref{eqn:lower} satisfies
\begin{equation} \label{eqn:lower-v}
    \sum_{j = 1}^m \min \bigl( \dist_{\filtrembed(\torus)}(\filtrembed(\mathbf{x}), \filtrembed(\mathbf{y})), q^2 \shvec(\filtrembed(\lattice)) \bigr)^2 \ge c_E \cdot \normsq{\mathbf{x} - \mathbf{y} - \mathbf{v}}
    \;\text.
\end{equation}

Let $j_1 \in \{0, 1, \dots, m\}$ be the smallest index satisfying that for all $j \in \{j_1 + 1, \dots, m\}$, $\filtrembed(\mathbf{v})$ realizes $\dist_{\filtrembed(\torus)}(\filtrembed(\mathbf{x}), \filtrembed(\mathbf{y}))$.
Then for all $j \in \{j_1 + 1, \dots, m\}$,
\begin{equation} \label{eqn:lower-dist}
    \begin{aligned}
        \dist_{\filtrembed(\torus)}(\filtrembed(\mathbf{x}), \filtrembed(\mathbf{y}))
        &= \norm{\filtrembed(\mathbf{x} - \mathbf{y} - \mathbf{v})} \\
        &\ge \norm{\filtrproj(\mathbf{x} - \mathbf{y} - \mathbf{v})}
        \;\text.
    \end{aligned}
\end{equation}
Moreover, according to Lemma~\ref{lem:Voronoi} and Corollary~\ref{cor:filtr},
\begin{equation} \label{eqn:lower-proj}
    \norm{\filtrproj(\mathbf{x} - \mathbf{y} - \mathbf{v})}
    \le \norm{\filtrprefixeq(\mathbf{x} - \mathbf{y} - \mathbf{v})}
    \le \cover(\lattice_j)
    \le q \shvec(\filtrembed(\lattice)) \le q^2 \shvec(\filtrembed(\lattice))
    \;\text.
\end{equation}
Combining \eqref{eqn:lower-dist} and \eqref{eqn:lower-proj}, the left-hand side of \eqref{eqn:lower-v} is bounded from below by
\begin{equation} \label{eqn:lower-sum}
    \sum_{j = j_1 + 1}^m \normsq{\filtrproj(\mathbf{x} - \mathbf{y} - \mathbf{v})}
    = \normsq{\filtrsuffixne[j_1](\mathbf{x} - \mathbf{y} - \mathbf{v})}
    \;\text.
\end{equation}
If it is the case that
\begin{equation*}
    \normsq{\filtrsuffixne[j_1](\mathbf{x} - \mathbf{y} - \mathbf{v})}
    \ge \frac{1}{2} \normsq{\mathbf{x} - \mathbf{y} - \mathbf{v}}
    \;\text,
\end{equation*}
then \eqref{eqn:lower-sum} clearly suffices to prove \eqref{eqn:lower-v}.
So from now on we assume that
\begin{align} \label{eqn:case-condition}
    \normsq{\filtrsuffixne[j_1](\mathbf{x} - \mathbf{y} - \mathbf{v})}
    &< \frac{1}{2} \normsq{\mathbf{x} - \mathbf{y} - \mathbf{v}}
    \;\text, \notag \\ \text{i.e., }
    \normsq{\filtrprefixeq[j_1](\mathbf{x} - \mathbf{y} - \mathbf{v})}
    &> \frac{1}{2} \normsq{\mathbf{x} - \mathbf{y} - \mathbf{v}}
    \;\text.
\end{align}
In particular, $j_1 > 0$.
Then, by definition, $\filtrembed[j_1](\mathbf{v})$ does not realize $\dist_{\filtrembed[j_1](\torus)}(\filtrembed[j_1](\mathbf{x}), \filtrembed[j_1](\mathbf{y}))$, which, by using Lemma~\ref{lem:closest-change} with lattice $\filtrembed[j_1](\lattice)$, implies
\begin{equation} \label{eqn:lower-change}
    \norm{\filtrembed[j_1](\mathbf{x} - \mathbf{y} - \mathbf{v})} \ge \frac{1}{2} \shvec(\filtrembed[j_1](\lattice))
    \;\text.
\end{equation}

Under assumption \eqref{eqn:case-condition}, it suffices to prove that there exists an index $j_0 \in \{1, \dots, m\}$ such that
\begin{equation*}
    \min \bigl( \dist_{\filtrembed[j_0](\torus)}(\filtrembed[j_0](\mathbf{x}), \filtrembed[j_0](\mathbf{y})), q^2 \shvec(\filtrembed[j_0](\lattice)) \bigr)^2 \ge c \cdot \normsq{\filtrprefixeq[j_1](\mathbf{x} - \mathbf{y} - \mathbf{v})}
    \;\text,
\end{equation*}
or equivalently,
\begin{align} \label{eqn:lower-item-dist}
    \normsq{\filtrembed[j_0](\mathbf{x} - \mathbf{y} - \mathbf{v}')} &\ge c \cdot \normsq{\filtrprefixeq[j_1](\mathbf{x} - \mathbf{y} - \mathbf{v})}
    \;\text, \ \text{and} \\ \label{eqn:lower-item-shortest}
    q^4 \shvec^2(\filtrembed[j_0](\lattice)) &\ge c \cdot \normsq{\filtrprefixeq[j_1](\mathbf{x} - \mathbf{y} - \mathbf{v})}
    \;\text,
\end{align}
where $\mathbf{v}' \in \lattice$ is a lattice point such that $\filtrembed[j_0](\mathbf{v}')$ realizes $\dist_{\filtrembed[j_0](\torus)}(\filtrembed[j_0](\mathbf{x}), \filtrembed[j_0](\mathbf{y}))$,
and $c$ is some absolute constant.
Note that, without loss of generality, it can be assumed that
\begin{equation} \label{eqn:lower-item-covering}
    \norm{\filtrprefixne[j_0](\mathbf{x} - \mathbf{y} - \mathbf{v}')} \le \cover(\lattice_{j_0-1})
\end{equation}
(since otherwise, we can use $\mathbf{v}' + \mathbf{u}$ instead of $\mathbf{v}'$, where $\mathbf{u} \in \lattice_{j_0-1}$ realizes $\dist(\filtrprefixne[j_0](\mathbf{x} - \mathbf{y} - \mathbf{v}'), \lattice_{j_0-1})$).

We choose $j_0 = j_1$ if
\begin{equation*}
    \cover(\lattice_{j_1-1}) \le \frac{1}{4} \shvec(\filtrembed[j_1](\lattice))
    \;\text,
\end{equation*}
and otherwise $j_0 = j_1 - 1$.
By Corollary~\ref{cor:filtr} and the condition $q \le \gamma^2 / 32$, we know that
\begin{equation*}
    \cover(\lattice_{j_1-2})
    \le \frac{q}{\gamma^2} \cdot \shvec(\filtrembed[j_1](\lattice))
    \le \frac{1}{32} \shvec(\filtrembed[j_1](\lattice))
    \;\text.
\end{equation*}
Moreover, as $\lattice_{j_1} / \lattice_{j_1-1}$ is both a quotient of $\lattice_{j_1}$ and a sublattice of $\filtrembed[j_1](\lattice)$,
\begin{equation*}
    \cover(\lattice_{j_1})
    \ge \cover(\lattice_{j_1} / \lattice_{j_1-1})
    \ge \frac{1}{2} \shvec(\lattice_{j_1} / \lattice_{j_1-1})
    \ge \frac{1}{2} \shvec(\filtrembed[j_1](\lattice))
\end{equation*}
(and the last inequality is actually an equality due to Lemma~\ref{lem:shortest-untouched}).
Therefore $j_0$ satisfies
\begin{align} \label{eqn:lower-covering-le}
    \cover(\lattice_{j_0-1}) &\le \frac{1}{4} \shvec(\filtrembed[j_1](\lattice))
    \;\text{, and} \\ \label{eqn:lower-covering-gt}
    \cover(\lattice_{j_0}) &> \frac{1}{4} \shvec(\filtrembed[j_1](\lattice))
    \;\text.
\end{align}

We first prove~\eqref{eqn:lower-item-shortest} for this choice of $j_0$:
\begin{equation*}
    \begin{aligned}
        q^2 \shvec(\filtrembed[j_0](\lattice))
        &\ge q \cover(\lattice_{j_0}) \\
        &> \frac{q}{4} \cdot \lambda_1(\filtrembed[j_1](\lattice)) \\
        &\ge \frac{1}{4} \cover(\lattice_{j_1}) \\
        &\ge \frac{1}{4} \norm{\filtrprefixeq[j_1](\mathbf{x} - \mathbf{y} - \mathbf{v})}
        \;\text,
    \end{aligned}
\end{equation*}
where the first and third inequalities use  Corollary~\ref{cor:filtr}, the second inequality follows from \eqref{eqn:lower-covering-gt}, and the last inequality uses Lemma~\ref{lem:Voronoi}.

We next prove~\eqref{eqn:lower-item-dist} for this choice of $j_0$.
We begin with showing that
\begin{equation} \label{eqn:lower-eq}
    \filtrsuffixne[j_2](\mathbf{v}') = \filtrsuffixne[j_2](\mathbf{v})
    \;\text,
\end{equation}
where $j_2 = \min(j_1 + 1, m)$.
Suppose towards contradiction that $\filtrsuffixne[j_2](\mathbf{v}') \ne \filtrsuffixne[j_2](\mathbf{v})$ (implying $j_2 < m$, and thus $j_2 = j_1 + 1$). Then, by definition, $\norm{\filtrembed[j_1+2](\mathbf{v} - \mathbf{v}')} \ge \shvec(\filtrembed[j_1+2](\lattice))$.
Hence
\begin{equation*}
    \begin{aligned}
        \norm{\filtrprefixeq[j_1](\mathbf{x} - \mathbf{y} - \mathbf{v})}
        &> \frac{1}{2} \norm{\mathbf{x} - \mathbf{y} - \mathbf{v}} \\
        &\ge \frac{1}{2}  \norm{\filtrembed[j_0](\mathbf{x} - \mathbf{y} - \mathbf{v})} \\
        &\ge \frac{1}{4} \norm{\filtrembed[j_0](\mathbf{v} - \mathbf{v}')} \\
        &\ge \frac{1}{4} \norm{\filtrsuffixeq[j_1+2](\filtrembed[j_0](\mathbf{v} - \mathbf{v}'))} \\
        &= \frac{\alpha^{j_1 - j_0 + 2}}{4} \norm{\filtrembed[j_1+2](\mathbf{v} - \mathbf{v}')} \\
        &\ge \frac{\alpha^{j_1 - j_0 + 2}}{4} \shvec(\filtrembed[j_1+2](\lattice))
        \;\text,
    \end{aligned}
\end{equation*}
where the first inequality follows from \eqref{eqn:case-condition} and the third inequality uses Lemma~\ref{lem:closest-change} with lattice $\filtrembed[j_0](\lattice)$.
On the other hand, we know that $\norm{\filtrprefixeq[j_1](\mathbf{x} - \mathbf{y} - \mathbf{v})} \le \cover(\lattice_{j_1})$ according to Lemma~\ref{lem:Voronoi}.
Then we have
\begin{equation} \label{eqn:lower-contradict}
    \frac{\alpha^{j_1 - j_0 + 2}}{4} \shvec(\filtrembed[j_1+2](\lattice))
    < \cover(\lattice_{j_1}) \le \frac{q}{\gamma^2} \shvec(\filtrembed[j_1+2](\lattice))
    \;\text,
\end{equation}
where the last inequality uses Corollary~\ref{cor:filtr}.
Since $\alpha^{j_1 - j_0 + 2} \ge \alpha^3 \ge 1/8$,
\eqref{eqn:lower-contradict} contradicts the condition $q \le \gamma^2 / 32$.

Based on \eqref{eqn:lower-eq}, we continue to prove \eqref{eqn:lower-item-dist} with the following observation:
\begin{align}
    \normsq{\filtrembed[j_0](\mathbf{x} - \mathbf{y} - \mathbf{v}')}
    &= \sum_{i = j_0}^{m} \alpha^{2(i - j_0)} \normsq{\filtrproj[i](\mathbf{x} - \mathbf{y} - \mathbf{v}')}
    \notag \\
    &\ge \alpha^{2(j_2 - j_0)} \sum_{i = j_0}^{j_2} \normsq{\filtrproj[i](\mathbf{x} - \mathbf{y} - \mathbf{v}')}
    \notag \\
    &= \alpha^{2(j_2 - j_0)} (\normsq{\mathbf{x} - \mathbf{y} - \mathbf{v}'} - \normsq{\filtrprefixne[j_0](\mathbf{x} - \mathbf{y} - \mathbf{v}')} - \normsq{\filtrsuffixne[j_2](\mathbf{x} - \mathbf{y} - \mathbf{v}')})
    \notag \\ \label{eqn:lower-item-dist-additive}
    &\ge \alpha^{2(j_2 - j_0)} (\normsq{\mathbf{x} - \mathbf{y} - \mathbf{v}} - \cover^2(\lattice_{j_0-1}) - \normsq{\filtrsuffixne[j_2](\mathbf{x} - \mathbf{y} - \mathbf{v})})
    \;\text,
\end{align}
where the last inequality uses the following three facts:
(i) as $\mathbf{v}$ realizes $\dist_{\torus}(\mathbf{x}, \mathbf{y})$, $\norm{\mathbf{x} - \mathbf{y} - \mathbf{v}'} \ge \norm{\mathbf{x} - \mathbf{y} - \mathbf{v}}$;
(ii) the term $\norm{\filtrprefixne[j_0](\mathbf{x} - \mathbf{y} - \mathbf{v}')}$ is bounded from above by \eqref{eqn:lower-item-covering};
and (iii) $\filtrsuffixne[j_2](\mathbf{x} - \mathbf{y} - \mathbf{v}') = \filtrsuffixne[j_2](\mathbf{x} - \mathbf{y} - \mathbf{v})$ due to \eqref{eqn:lower-eq}.
Moreover, according to \eqref{eqn:lower-covering-le} and \eqref{eqn:lower-change},
\begin{equation*}
    \begin{aligned}
        \cover(\lattice_{j_0-1})
        &\le \frac{1}{4} \shvec(\filtrembed[j_1](\lattice)) \\
        &\le \frac{1}{2} \norm{\filtrembed[j_1](\mathbf{x} - \mathbf{y} - \mathbf{v})} \\
        &\le \frac{1}{2} \norm{\mathbf{x} - \mathbf{y} - \mathbf{v}}
        \;\text,
    \end{aligned}
\end{equation*}
and according to \eqref{eqn:case-condition},
\begin{equation*}
    \begin{aligned}
        \normsq{\filtrsuffixne[j_2](\mathbf{x} - \mathbf{y} - \mathbf{v})}
        &\le \normsq{\filtrsuffixne[j_1](\mathbf{x} - \mathbf{y} - \mathbf{v})} \\
        &< \frac{1}{2} \normsq{\mathbf{x} - \mathbf{y} - \mathbf{v}}
        \;\text.
    \end{aligned}
\end{equation*}
Hence~\eqref{eqn:lower-item-dist-additive} is further bounded from below by
\begin{equation*}
    \begin{aligned}
        \alpha^{2(j_2 - j_0)} \Bigl( 1 - \frac{1}{4} - \frac{1}{2} \Bigr) \normsq{\mathbf{x} - \mathbf{y} - \mathbf{v}}
        &\ge \frac{\alpha^4}{4} \normsq{\mathbf{x} - \mathbf{y} - \mathbf{v}} \\
        &\ge \frac{\alpha^4}{4} \normsq{\filtrprefixeq[j_1](\mathbf{x} - \mathbf{y} - \mathbf{v})}
        \;\text.
    \end{aligned}
\end{equation*}
This completes the proof of~\eqref{eqn:lower-item-dist}, and the proof of the lemma.
\end{proof}

\renewcommand{\filtrproj}{\verbosefiltrproj}
\renewcommand{\filtrsuffixeq}{\verbosefiltrsuffixeq}
\renewcommand{\filtrprefixne}{\verbosefiltrprefixne}
\renewcommand{\filtrsuffixne}{\verbosefiltrsuffixne}
\renewcommand{\filtrprefixeq}{\verbosefiltrprefixeq}
\renewcommand{\filtrembed}{\verbosefiltrembed}
\renewcommand{\fullfiltrembed}{\verbosefullfiltrembed}

\subsection{Summary of Embedding into Tori}\label{composition}

By Lemma~\ref{lem:existsgoodfilter}, for any lattice $\lattice$, there exists an $(n\sqrt{n},n)$-filtration\footnote{This choice of filtration actually only shows the Lemma \ref{lem:filtration-distortion} for sufficiently large $n$. Choosing a $(32n\sqrt{n},32n)$-filtration gives us the lemma for all $n \ge 1$.} $\filtr$ of $\lattice$. Applying Lemmas~\ref{lem:filtration-upper} and \ref{lem:filtration-lower} to the embedding $\fullfiltrembed$ with $\alpha = 1/2$, we have the following.

\begin{lemma} \label{lem:filtration-distortion}
For any sufficiently large $n \ge 1$ and lattice $\lattice \subseteq \R^n$,
there exists $m \ge 1$ and embedding $\finalfullfiltrembed = (\finalfiltrembed[1], \dots, \finalfiltrembed[m])$ such that each $\finalfiltrembed$ maps the torus $\torus$ to some other torus, and $\finalfullfiltrembed$ satisfies
\begin{gather*}
    \sum_{j = 1}^m \dist_{\finalfiltrembed(\torus)}(\finalfiltrembed(\mathbf{x}), \finalfiltrembed(\mathbf{y}))^2
    \le
    c_{E, u} \cdot \dist_{\torus}(\mathbf{x}, \mathbf{y})^2
    \;\text{, and} \\
    \sum_{j = 1}^m \min \bigl( \dist_{\finalfiltrembed(\torus)}(\finalfiltrembed(\mathbf{x}), \finalfiltrembed(\mathbf{y})), p(n) \shvec(\finalfiltrembed(\lattice)) \bigr)^2 \ge c_{E, l} \cdot \dist_{\torus}(\mathbf{x}, \mathbf{y})^2
    \;\text,
\end{gather*}
where $c_{E, u}$ and $c_{E, l}$ are positive absolute constants and $p(n)$ is a fixed polynomial.
\end{lemma}

\section{Putting it All Together}

{
\def\thetheorem{\ref{thm:main}}
\begin{theorem}
For any lattice $\lattice \subseteq \R^n$ there exists a metric embedding of $\torus$ into Hilbert space with distortion $O(\sqrt{n\log n})$.
\end{theorem}
\addtocounter{theorem}{-1}
}

\begin{proof}
It suffices to show the embedding for sufficiently large $n$ (by, say, using the embedding from~\cite{KhotN06} for small $n$). Consider the composition
\begin{equation*}
    \Bigl(
    \finalgaussembed[{\finalfiltrembed[1](\lattice)}] \circ \finalfiltrembed[1], \dots, \finalgaussembed[{\finalfiltrembed[m](\lattice)}] \circ \finalfiltrembed[m]
    \Bigr)
    \;\text,
\end{equation*}
where $(\finalfiltrembed[1], \dots, \finalfiltrembed[m])$ is the embedding provided by Lemma~\ref{lem:filtration-distortion}. Let $k = \lceil \log_2 p(n) \rceil+1$ (where $p(n)$ is the fixed polynomial in Lemma~\ref{lem:filtration-distortion}).
By Lemma~\ref{lem:gaussian-distortion} and Lemma~\ref{lem:filtration-distortion}, noting that the modified contraction properties in both match, it follows immediately that the composed embedding has distortion at most
\begin{equation*}
    \sqrt\frac{\pi k n \cdot c_{E, u}}{c_H \cdot c_{E, l}}
    \;\text,
\end{equation*}
where $c_H$, $c_{E, u}$ and $c_{E, l}$ are all absolute constants.
Note that $k = \Theta(\log n)$.
Hence the distortion of the composed embedding is $O(\sqrt{n \log n})$.
\end{proof}


\bibliography{main}{}
\bibliographystyle{alpha}

\end{document}